\documentclass[11pt, letterpaper]{amsart}
\usepackage{fullpage,amsfonts,amsmath,amssymb,amsthm}
\usepackage{hyperref}
\usepackage{enumitem}
\newtheorem{thm}{Theorem}[section]
\newtheorem{cor}[thm]{Corollary}
\newtheorem{lem}[thm]{Lemma}

\theoremstyle{definition}

\newtheorem*{remark}{Remark}

\newtheoremstyle{TheoremNum}
        {\topsep}{\topsep}              
        {\itshape}                      
        {}                              
        {\bfseries}                     
        {.}                             
        { }                             
        {\thmname{#1}\thmnote{ \bfseries #3}}
    \theoremstyle{TheoremNum}
    \newtheorem{thmn}{Theorem}
\numberwithin{equation}{section}
\begin{document}
\title{Supercharacters of unipotent groups defined by involutions}
\author{Scott Andrews}
\address{Department of Mathematics \\ University of Colorado Boulder \\ Boulder, CO 80309}
\email{scott.andrews@colorado.edu}

\begin{abstract} We construct supercharacter theories of finite unipotent groups in the orthogonal, symplectic and unitary types. Our method utilizes group actions in a manner analogous to that of Diaconis and Isaacs in their construction of supercharacters of algebra groups. The resulting supercharacter theories agree with those of Andr\'e and Neto in the case of the unipotent orthogonal and symplectic matrices and generalize to a large collection of subgroups. In the unitary group case, we describe the supercharacters and superclasses in terms of labeled set partitions and calculate the supercharacter table.

\end{abstract}

\maketitle

\section{Introduction}
For $q$ a power of a prime, let $UT_n(\mathbb{F}_q)$ denote the group of unipotent $n \times n$ upper triangular matrices over the finite field with $q$ elements. Classifying the irreducible representations of $UT_n(\mathbb{F}_q)$ is known to be a ``wild'' problem (see \cite{MR1056208}). In \cite{MR1896026}, Andr\'e constructs a set of characters, referred to as ``basic characters,'' such that each irreducible character of $UT_n(\mathbb{F}_q)$ occurs with nonzero multiplicity in exactly one basic character. These characters can be thought of as a coarser approximation of the irreducible characters of $UT_n(\mathbb{F}_q)$. Diaconis--Isaacs generalize the idea of a basic character to a ``supercharacter'' of an arbitrary finite group in \cite{MR2373317}. They also construct supercharacter theories for all finite algebra groups $G$, which are subgroups of $UT_n(\mathbb{F}_q)$ such that $\{g-1 \mid g \in G\}$ is an $\mathbb{F}_q$-algebra. In the case that $G = UT_n(\mathbb{F}_q)$, the constructions of Andr\'e and of Diaconis--Isaacs produce the same supercharacter theory. The two constructions use different techniques; Andr\'e constructs basic characters by inducing linear characters from certain subgroups of $UT_n(\mathbb{F}_q)$, whereas Diaconis--Isaacs utilize the two-sided action of $UT_n(\mathbb{F}_q)$ on the associative algebra of strictly upper triangular matrices.

\bigbreak

Andr\'e--Neto have modified Andr\'e's earlier construction to the unitriangular groups in types $B,C$ and $D$ in \cite{MR2264135,MR2537684,MR2457229}. In this paper, we generalize these supercharacter theories in a manner analogous to the type $A$ construction of Diaconis--Isaacs. The construction in \cite{MR2264135,MR2537684,MR2457229} uses the idea of a ``basic subset of roots'' to induce linear characters from certain subgroups of the full unitriangular group.  Our construction instead utilizes actions of $UT_n(\mathbb{F}_q)$ on the Lie algebras of the unitriangular groups in types $B,C$ and $D$ to define superclasses and supercharacters. One advantage of our method is that it works in situations where the idea of a basic subset of roots does not make sense, such as the case of the unipotent radical of a parabolic subgroup.

\bigbreak

Aguiar et al. construct a Hopf algebra on the type $A$ supercharacters in \cite{MR2880223} and show that this structure is isomorphic to the Hopf algebra of symmetric functions in non-commuting variables. In \cite{MR3119357}, Benedetti has constructed an analogous Hopf algebra on the superclass functions of type $D$. Marberg describes the type $B$ and $D$ supercharacters in terms of type $A$ supercharacters in \cite{MR2880659}. We hope that our construction will allow for many more type $A$ results to be generalized to other types.

\bigbreak

Given a pattern subgroup $G$ of $UT_n(\mathbb{F}_q)$ (an algebra group such that $\{g-1\mid g\in G\}$ has a basis of elementary matrices) and a subgroup $U$ of $G$ defined by an anti-involution of $G$, we construct a supercharacter theory. The anti-involution of $G$ induces an action of $G$ on the Lie algebra of $U$, which we use to construct the superclasses and supercharacters. The examples that naturally fall into this context include the unipotent orthogonal, symplectic, and unitary groups. Let $J$ denote the $n \times n$ matrix with ones on the anti-diagonal and zeroes elsewhere; for $q$ a power of an odd prime, define
\begin{align*}
        UO_n(\mathbb{F}_q) & = \{g \in UT_n(\mathbb{F}_q) \mid g^{-1} = Jg^tJ\} \text{ \quad and}\\
        USp_{2n}(\mathbb{F}_q) & = \bigg\{g \in UT_{2n}(\mathbb{F}_q) \;\bigg|\; g^{-1} = -\left( \begin{array}{cc} 0&J\\-J&0\end{array}\right) g^t \left( \begin{array}{cc} 0&J\\-J&0\end{array}\right)\bigg\}. \\
\end{align*}
The groups $UO_n(\mathbb{F}_q)$ are the unipotent groups of types $B$ and $D$, and the groups $USp_{2n}(\mathbb{F}_q)$ are the unipotent groups of type $C$. Note that these groups are each defined by an anti-involution of $UT_n(\mathbb{F}_q)$; our construction produces the supercharacter theories constructed by Andr\'e--Neto in \cite{MR2264135,MR2537684,MR2457229}.

\bigbreak

We can also construct supercharacter theories of the unipotent unitary groups. For $g \in UT_n(\mathbb{F}_{q^2})$, define $\overline{g}$ by $(\overline{g})_{ij} = (g_{ij})^q$. Let
\[
        UU_n(\mathbb{F}_{q^2}) = \{g \in UT_n(\mathbb{F}_{q^2}) \mid g^{-1} = J\overline{g}^tJ \}.
\]
The group $UU_n(\mathbb{F}_{q^2})$ is the group of unipotent unitary $n \times n$ matrices over $\mathbb{F}_{q^2}$. As $UU_n(\mathbb{F}_{q^2})$ is a subgroup of $UT_n(\mathbb{F}_{q^2})$ that is defined by an anti-involution, we get a supercharacter theory from the action of $UT_n(\mathbb{F}_{q^2})$ on the Lie algebra of $UU_n(\mathbb{F}_{q^2})$. The supercharacter values on superclasses demonstrate \emph{Ennola duality}, as they are obtained from the supercharacter values of $UT_n(\mathbb{F}_q)$ by formally replacing `$q$' with `$-q$'.

\bigbreak

We present the main result of the paper in Section~\ref{mainresult}, which is applied in Section~\ref{examples} to construct supercharacter theories of unipotent orthogonal and symplectic groups. We develop necessary background material on the interactions between groups, algebras and vector spaces in Section~\ref{background}. We review the construction of supercharacter theories of algebra groups in Section~\ref{alggpsct}, and prove our main result in Section~\ref{mainresultproof}. Finally, in Section~\ref{unitary}, we construct supercharacter theories of the unipotent unitary groups and calculate the values of supercharacters on superclasses.

\section{Main result}\label{mainresult}

The main result of this paper is the construction of a supercharacter theory for certain subgroups of algebra groups that are defined by anti-involutions. In this section we review the necessary background material on algebra groups and present the main result of the paper.

\subsection{Supercharacter theories}\label{sctig}

The idea of a supercharacter theory of an arbitrary finite group was introduced by Diaconis--Isaacs in \cite{MR2373317}, and has been connected to a number of areas of mathematics. In \cite{MR2989654}, Hendrickson shows that the supercharacter theories of a finite group $G$ are in bijection with the central Schur rings over $G$. Brumbaugh et al. construct certain exponential sums of interest in number theory (e.g., Gauss, Ramanujan, and Kloosterman sums) as supercharacters of abelian groups in \cite{MR3239156}. For a more in-depth treatment of supercharacters see \cite{MR2373317}; we only address the basics that are necessary for our construction.

\bigbreak

Let $G$ be a finite group, and suppose that $\mathcal{K}$ is a partition of $G$ into unions of conjugacy classes and $\mathcal{X}$ is a set of characters of $G$. We say that the pair $(\mathcal{K},\mathcal{X})$ is a \emph{supercharacter theory} of $G$ if
\begin{enumerate}
\item $|\mathcal{X}|=|\mathcal{K}|$,
\item the characters $\chi\in \mathcal{X}$ are constant on the members of $\mathcal{K}$, and
\item each irreducible character of $G$ is a constituent of exactly one character in $\mathcal{X}$.
\end{enumerate}

The characters $\chi \in \mathcal{X}$ are referred to as \textit{supercharacters} and the sets $K \in \mathcal{K}$ are called $\textit{superclasses}$.

\subsection{Algebra groups and pattern subgroups}

Let $\mathbb{F}$ be a field and let $\frak{g}$ be a nilpotent associative algebra over $\mathbb{F}$. The \textit{algebra group} $G$ associated to $\frak{g}$ is the set of formal sums
\[
        G = \{1+x \mid x \in \frak{g}\}
\]
with multiplication defined by $(1+x)(1+y) = 1+(x+y+xy)$  (see \cite{MR1358482}). As $\frak{g}$ is nilpotent, elements in $G$ have inverses given by
\[
        (1+x)^{-1} = 1 + \sum_{i=1}^\infty (-x)^i.
\]
We will often write $G = 1+\frak{g}$ to indicate that $G$ is the algebra group associated to $\frak{g}$.

\bigbreak

For example, if we define $UT_n(\mathbb{F}_q)$ to be the group of $n \times n$ upper triangular matrices over $\mathbb{F}_q$ with ones on the diagonal and $\frak{ut}_n(\mathbb{F}_q)$ to be the algebra of $n \times n$ upper triangular matrices over $\mathbb{F}_q$ with zeroes on the diagonal, then $UT_n(\mathbb{F}_q)$ is the algebra group associated to $\frak{ut}_n(\mathbb{F}_q)$.

\bigbreak

Let $\mathcal{P}$ be a poset on $[n]$ that is a sub-order of the usual linear order. In other words, $\mathcal{P}$ has the properties that
\begin{enumerate}
\item if $i \preceq_\mathcal{P} j$ then $i \leq j$,
\item if $i \preceq_\mathcal{P} j$ and $j \preceq_\mathcal{P} k$ then $i \preceq_\mathcal{P} k$, and
\item $i \preceq_\mathcal{P} i$ for all $i \in [n]$.
\end{enumerate}
Corresponding to the poset $\mathcal{P}$ are a \textit{pattern subgroup}
\[
        U_\mathcal{P} = \{g \in UT_n(\mathbb{F}_q) \mid g_{ij} = 0 \text{ unless } i \preceq_\mathcal{P} j\}
\]
and a \textit{pattern subalgebra}
\[
        \frak{u}_\mathcal{P}  = \{x \in \frak{ut}_n(\mathbb{F}_q) \mid x_{ij} = 0 \text{ unless } i \prec_\mathcal{P} j\}.
\]
Note that $U_\mathcal{P}$ is the algebra group corresponding to $\frak{u}_\mathcal{P}$. For a more complete discussion of pattern subgroups, see \cite{MR2491890}.

\bigbreak

In \cite{MR2373317}, Diaconis--Isaacs construct a supercharacter theory for an arbitrary finite algebra group $G = 1+\frak{g}$. Note that $G$ acts on $\frak{g}$ by left and right multiplication; there are corresponding actions of $G$ on the dual $\frak{g}^*$ given by
\[
        (g \lambda)(x) = \lambda(g^{-1}x) \quad \text{ and } (\lambda g)(x) = \lambda(xg^{-1}),
\]
where $g \in G$, $\lambda \in \frak{g}^*$, and $x \in \frak{g}$. Let
\begin{align*}
        f:G &\to \frak{g} \\
        g &\mapsto g-1
\end{align*}
and let $\theta:F_q^+ \to \mathbb{C}^\times$ be a nontrivial homomorphism. For $g \in G$ and $\lambda \in \frak{g}^*$, define
\[
        K_g = \{h \in G \mid f(h) \in Gf(g)G\}
\]
and
\[
        \chi_\lambda =\frac{|G\lambda|}{|G\lambda G|} \sum_{\mu \in G\lambda G} \theta \circ \mu \circ f.
\]

\begin{thmn}[\ref{sctalggp} \textnormal{(\cite{MR2373317}})] The partition of $G$ given by $\mathcal{K} = \{K_g \mid g \in G\}$, along with the set of characters $\{\chi_\lambda \mid \lambda \in \frak{g}^*\}$, form a supercharacter theory of $G$. This supercharacter theory is independent of the choice of $\theta$.

\end{thmn}

\bigbreak

The supercharacter theory is independent of $\theta$ in that the sets $\mathcal{K}$ and $\{\chi_\lambda \mid \lambda \in \frak{u}^*\}$ do not depend on $\theta$. If a different $\theta$ is chosen, the $\chi_\lambda$ will be permuted. In Section~\ref{alggpsct} we present a modified proof of this result as motivation for the proof of our main result.

\subsection{Subgroups of algebra groups defined by anti-involutions}\label{subgroups}

For $q$ a power of a prime, let $\frak{g}$ be a nilpotent associative algebra of finite dimension over $\mathbb{F}_q$. Define $G = 1+\frak{g}$. We equip $\frak{g}$ with a Lie algebra structure given by $[x,y] = xy-yx$.

\bigbreak

Let
\begin{align*}
        \dagger: \frak{g} &\to \frak{g} \\
        x & \mapsto x^\dagger
\end{align*}
be an involutive associative algebra antiautomorphism, and for $x \in \frak{g}$ define $(1+x)^\dagger = 1 + x^\dagger$. Note that this makes $\dagger$ an involutive antiautomorphism of $G$.

\bigbreak

Define
\[
        U = \{u \in G \mid u^{\dagger} = u^{-1}\}
\]
and
\[
        \frak{u} = \{x \in \frak{g} \mid x^\dagger = -x\}.
\]
Note that $\frak{u}$ is not an associative algebra, although it is closed under the Lie bracket.

\bigbreak

For $g \in G$ and $x \in \frak{g}$, define $g \cdot x = gxg^\dagger$. It is routine to check that this defines a linear action of $G$ on $\frak{g}$. The action restricts to an action of $G$ on $\frak{u}$, and for $x \in \frak{g}$ and $u \in U$, $u \cdot x = uxu^{-1}$. We can also define a left action of $\frak{g}$ on itself by $x*y = xy + yx^\dagger$. This action restricts to an action of $\frak{g}$ on $\frak{u}$, and for $x \in \frak{u}$ and $y \in \frak{g}$, $x * y = [x,y]$.

\bigbreak

The motivating examples of groups defined in this manner are the unipotent orthogonal, symplectic, and unitary groups in odd characteristic. For instance, if $G = UT_n(\mathbb{F}_q)$ and $\frak{g} = \frak{ut}_n(\mathbb{F}_q)$, with $q$ odd, we can define an antiautomorphism
\begin{align*}
        \dagger: \frak{g} &\to \frak{g} \\
        x &\mapsto Jx^tJ
\end{align*}
where $J$ is the matrix with ones on the antidiagonal and zeroes elsewhere. Then
\[
        UO_n(\mathbb{F}_q) = \{u \in UT_n(\mathbb{F}_q) \mid u^\dagger = u^{-1}\}
\]
and
\[
        \frak{uo}_n(\mathbb{F}_q) = \{x \in \frak{ut}_n(\mathbb{F}_q) \mid x^\dagger = -x\}.
\]
The unipotent symplectic and unitary groups can be similarly described in terms of antiautomorphisms of the upper triangular matrices.

\subsection{Springer morphisms}

In order to utilize the Lie algebra structure of $\frak{u}$ to study $U$, we would like a bijection between $U$ and $\frak{u}$ that preserves useful properties. In the case of an algebra group $G$, we can use the map $g \mapsto g-1$ to relate $G$ to $\frak{g}$. In general, however, it is not the case that $U = 1 +\frak{u}$, so we need a variation on this map. Andr\'e--Neto define a bijection from $U$ to $\frak{u}$ in \cite{MR2264135}, however we require a map that is invariant under the adjoint action of $U$.

\bigbreak

Given an algebra group $G = 1+\frak{g}$ and a map $\dagger$ as above, we define a \textit{Springer morphism} $f:G \to \frak{g}$ to be a bijection such that
\begin{enumerate}
\item $f(U) = \frak{u}$ , and
\item there exist $a_i \in \mathbb{F}_q$ such that $f(1+x) = x+\sum_{i=2}^\infty a_ix^i$.
\end{enumerate}
The dependence of these conditions on $\dagger$ is implicit in that $U$ and $\frak{u}$ are defined in terms of $\dagger$. Note that condition (2) gives that $f(H) = \frak{h}$ for any algebra subgroup $H = \frak{h}+1$, and also guarantees that $f$ will be invariant under the adjoint action of $G$. We require that the coefficient of the $x$ term of $f(1+x)$ be 1 for ease of computation; relaxing this condition would not have any effect on the resulting supercharacter theory.

\bigbreak

Springer morphisms are introduced by Springer and Steinberg in \cite{MR0268192} (III, 3.12) and are utilized by Kawanaka in \cite{MR803335}. Our definition of a Springer morphism is slightly modified from the original definition, but the examples given below are Springer morphisms in the original sense. The logarithm map
\[
        f(1+x) = \sum_{i=1}^\infty (-1)^{i+1}\frac{x^i}{i}
\]
is perhaps the most natural choice of a Springer morphism, but is not defined in many characteristics. The map
\[
        h(1+x) = 2x(x+2)^{-1}.
\]
is, however, a Springer morphism in all odd characteristics. We mention that this is a constant multiple of the map $1+x \mapsto x(x+2)^{-1}$, which is often referred to as the \emph{Cayley map} (see, for instance, \cite{MR803335}). The following lemma is easy to verify directly.

\begin{lem} Let $q$ be a power of the prime $p$, and let $f$ and $h$ be the maps defined above. Let $G = 1+\frak{g}$ be any algebra group, and let $\dagger$ be any anti-involution of $\frak{g}$. If $x^p = 0$ for all $x \in \frak{g}$, then $f$ is a Springer morphism. If $p$ is odd, then $h$ is a Springer morphism.

\end{lem}

This lemma allows us to assume the existence of a Springer morphism if we are working in odd characteristic, which we will do for the remainder of the paper.

\subsection{Main theorem}\label{mainthm}

Let $q$ be a power of an odd prime, and let $G = 1+\frak{g}$ be a pattern subgroup of $UT_n(\mathbb{F}_{q^k})$ for some $n$ and $k$. For $1 \leq i \leq n$, define $\overline{i} = n+1-i$. We consider $\frak{g}$ as an $\mathbb{F}_q$-algebra; let $\dagger$ be an anti-involution of $\frak{g}$ such that $(\alpha e_{ij})^\dagger \in \mathbb{F}_{q^k}^\times e_{\bar{j}\bar{i}}$ for all $\alpha \in \mathbb{F}_{q^k}^\times$. In other words, $\dagger$ reflects the entries of elements of $\frak{g}$ across the antidiagonal, up to a constant multiple. The antiautomorphisms that define the orthogonal, symplectic and unitary groups all have this property. Let
\[
        U = \{u \in G \mid u^{\dagger} = u^{-1}\}
\]
and
\[
        \frak{u} = \{x \in \frak{g} \mid x^\dagger = -x\}.
\]
Let $f$ be any Springer morphism and let $\theta:\mathbb{F}_q^+ \to \mathbb{C}^\times$ be a nontrivial homomorphism. For $g \in G$, $x \in \frak{u}$ and $\lambda \in \frak{u}^*$, let $g \cdot x = gxg^\dagger$ and $(g \cdot \lambda)(x) = \lambda(g^{-1} \cdot x)$. For $\lambda \in \frak{u}^*$ and $u \in U$, define
\begin{equation}\label{eqku}
        K_u = \{ v \in U \mid f(v) \in G \cdot f(u)\}
\end{equation}
and
\begin{equation}\label{eqchilambda}
        \chi_\lambda = \frac{1}{n_\lambda} \sum_{\mu \in G \cdot \lambda}
        \theta \circ \mu \circ f,
\end{equation}
where $n_\lambda$ is a constant determined by $\lambda$ (and independent of the choice of $\lambda$ as orbit representative). As in \cite{MR2373317}, $n_\lambda$ can be written in terms of the sizes of orbits of group actions. If we let $H$ be the subgroup of $G$ defined by
\[
H = \bigg\{ h\in G \:\bigg|\: h_{ij} = 0 \text{ if } j\leq\frac{n}{2}\bigg\},
\]
then
\begin{equation}\label{nlambda}
n_\lambda = \frac{|G \cdot \lambda|}{|H \cdot \lambda|}.
\end{equation}

\begin{thmn}[\ref{sctofu}] The partition of $U$ given by $\mathcal{K} = \{K_u \mid u \in U\}$, along with the set of characters $\{\chi_\lambda \mid \lambda \in \frak{u}^*\}$, form a supercharacter theory of $U$. This supercharacter theory is independent of the choice of $\theta$ and $f$.

\end{thmn}

The supercharacter theory is independent of $\theta$ and $f$ in that the sets $\mathcal{K}$ and $\{\chi_\lambda \mid \lambda \in \frak{u}^*\}$ do not depend on these functions. If a different $\theta$ is chosen or condition (2) in the definition of a Springer morphism is relaxed to allow for other $x$ coefficients, the $\chi_\lambda$ will be permuted. The supercharacter theory is also independent of the choice of subfield of $\mathbb{F}_{q^k}$; that is, if $\mathbb{F}$ is any subfield of $\mathbb{F}_{q^k}$ and $\dagger$ is an antiautomorphism of $\frak{g}$ when viewed as an $\mathbb{F}$-algebra, we get the same supercharacter theory as by considering $\frak{g}$ as an $\mathbb{F}_q$-algebra.  We will prove this theorem in Section~\ref{mainresultproof}, along with the following result that allows us to relate our supercharacter theories to those of Andr\'e--Neto.

\begin{thmn}[\ref{superclassintersect}]
    The superclasses of $U$ are exactly the sets of the form $U \cap K_g$, where $K_g$ is some superclass of $G$.
\end{thmn}

\begin{remark} Note that the superclasses of $G$ are determined by an action of $G \times G$ on $\frak{g}$ (with one $G$ acting on each side), whereas the superclasses of $U$ only require a left action of $G$ on $\frak{u}$. This may seem strange, especially in light of Theorem~\ref{superclassintersect}. The reason that we only need one copy of $G$ to act on $\frak{u}$ is due to the fact that if $x \in \frak{u}$ and $gx \in \frak{u}$, then there exists $h \in G$ with $gx = hxh^\dagger$. In other words, because the elements of $\frak{u}$ respect an involution we only need one copy of $G$ acting on the left to construct the superclasses. For the details, see the proof of Theorem~\ref{superclassintersect}.
\end{remark}

\section{Supercharacter theories of unipotent orthogonal and symplectic groups}\label{examples}

Before we prove Theorem~\ref{sctofu}, we use it in this section to construct supercharacter theories for two families of groups.

\subsection{Supercharacter theories of unipotent orthogonal groups}

Let $J$ be the $n \times n$ matrix with ones on the antidiagonal and zeroes elsewhere, and let $x^t$ denote the transpose of a matrix $x$. For $q$ a power of an odd prime, define
\[
        O_{n}(\mathbb{F}_q) = \{g \in GL_{n}(\mathbb{F}_q) \mid g^{-1} = J g^t J\}
\]
along with the corresponding Lie algebra
\[
        \frak{o}_{n}(\mathbb{F}_q) = \{x \in \frak{gl}_{n}(\mathbb{F}_q) \mid -x = J x^t J\}.
\]
Define
\begin{align*}
        UO_{n}(\mathbb{F}_q) &= UT_{n}(\mathbb{F}_q) \cap O_{n}(\mathbb{F}_q) \text{ and} \\
        \frak{uo}_{n}(\mathbb{F}_q) &= \frak{ut}_{n}(\mathbb{F}_q) \cap \frak{o}_{n}(\mathbb{F}_q).
\end{align*}
Define an antiautomorphism $\dagger$ of $\frak{ut}_{n}(\mathbb{F}_q)$ by $x^\dagger = J x^t J$. Note that $\dagger$ satisfies the conditions required by Theorem~\ref{sctofu}, and furthermore
\begin{align*}
        UO_{n}(\mathbb{F}_q) &= \{g \in UT_{n}(\mathbb{F}_q) \mid g^{-1} = g^\dagger\} \text{ and} \\
        \frak{uo}_{n}(\mathbb{F}_q) &= \{x \in \frak{ut}_{n}(\mathbb{F}_q) \mid -x = x^\dagger\}.
\end{align*}
Define $K_u$ and $\chi_\lambda$ as in \ref{eqku} and \ref{eqchilambda} with $U = UO_n(\mathbb{F}_q)$ and $\frak{u} = \frak{uo}_n(\mathbb{F}_q)$. By Theorem~\ref{sctofu}, there is a supercharacter theory of $UO_n(\mathbb{F}_q)$ with superclasses $\{K_u\}$ and supercharacters $\{\chi_\lambda\}$.

\bigbreak

In \cite{MR2537684}, Andr\'e--Neto construct a supercharacter theory of $UO_n(\mathbb{F}_q)$. They show that their superclasses are the sets of the form $UO_n(\mathbb{F}_q)\cap K_g$, where $K_g$ is a superclass of $UT_n(\mathbb{F}_q)$ under the algebra group supercharacter theory. In particular, the following theorem follows from Theorem~\ref{superclassintersect}.

\begin{thm} The supercharacter theory of $UO_n(\mathbb{F}_q)$ defined above coincides with that of Andr\'e--Neto in \cite{MR2537684}.
\end{thm}

\bigbreak

We can also construct supercharacter theories of certain subgroups of $UO_n(\mathbb{F}_q)$ using this method. We will call a poset $\mathcal{P}$ a \emph{mirror poset} if $i \preceq_\mathcal{P} j$ implies that $\bar{j} \preceq_\mathcal{P} \bar{i}$ (recall that $\bar{i} = n-i+1$). The antiautomorphism $\dagger$ as defined above restricts to an antiautomorphism of $U_\mathcal{P}$ for any mirror poset. Furthermore,
\begin{align*}
        UO_{n}(\mathbb{F}_q)\cap U_\mathcal{P} &= \{g \in U_\mathcal{P} \mid g^{-1} = g^\dagger\} \text{ and} \\
        \frak{uo}_{n}(\mathbb{F}_q)\cap \frak{u}_\mathcal{P} &= \{x \in \frak{u}_\mathcal{P} \mid -x = x^\dagger\}.
\end{align*}
Define $K_u$ and $\chi_\lambda$ as in \ref{eqku} and \ref{eqchilambda} with $U = UO_{n}(\mathbb{F}_q)\cap U_\mathcal{P}$ and $\frak{u} = \frak{uo}_{n}(\mathbb{F}_q)\cap \frak{u}_\mathcal{P}$. By Theorem~\ref{sctofu}, there is a supercharacter theory of $UO_{n}(\mathbb{F}_q)\cap U_\mathcal{P}$ with superclasses $\{K_u\}$ and supercharacters $\{\chi_\lambda\}$. By Theorem~\ref{superclassintersect}, the superclasses are of the form $K_g \cap UO_n(\mathbb{F}_q)$ where $K_g$ is a superclass of $U_\mathcal{P}$ in the algebra group supercharacter theory. In particular, if $U$ is the unipotent radical of a parabolic subgroup of $O_n(\mathbb{F}_q)$ then $U = UO_n(\mathbb{F}_q)\cap U_\mathcal{P}$ for some mirror poset $\mathcal{P}$.

\bigbreak

There are two important examples of a subgroup obtained from a mirror poset in type $D$. First, let $P$ be the mirror poset on $[2n]$ defined by
\[
        i \preceq_\mathcal{P} j \text{ if } i \leq j \text{ and } (i,j) \neq (n,n+1).
\]
Then $UO_{2n}(\mathbb{F}_q) \cap U_\mathcal{P} = UO_{2n}(\mathbb{F}_q)$, and we get a second supercharacter theory of $UO_{2n}(\mathbb{F}_q)$ which is at least as fine as the one originally defined. This new supercharacter theory is in fact strictly finer than the original; the elements $e_{1,n}-e_{n+1,2n}$ and $(e_{1,n}-e_{n+1,2n})+(e_{1,n+1}-e_{n,2n})$ of $\frak{u}$ are in the same orbit under the action of $UT_{2n}(\mathbb{F}_q)$ on $\frak{uo}_{2n}(\mathbb{F}_q)$, but in different orbits under the action of $U_\mathcal{P}$ on $\frak{uo}_{2n}(\mathbb{F}_q)$.

\bigbreak

We can also consider the poset $\mathcal{P}$ on $[2n]$ defined by
\[
        i \preceq_\mathcal{P} j \text{ if } i\leq j\leq n \text{ or } n+1 \leq i \leq j.
\]
In this case, $UO_{2n}(\mathbb{F}_q) \cap U_\mathcal{P} \cong UT_n(\mathbb{F}_q)$, and the supercharacter theory obtained is the algebra group supercharacter theory.

\subsection{Supercharacter theories of unipotent symplectic groups}

Define
\[
        \Omega = \left(\begin{array}{cc} 0 & -J \\ J & 0 \end{array}\right),
\]
where once again $J$ is the $n \times n$ matrix with ones on the antidiagonal and zeroes elsewhere. For $q$ a power of an odd prime, define
\[
        Sp_{2n}(\mathbb{F}_q) = \{g \in GL_{2n}(\mathbb{F}_q) \mid g^{-1} = -\Omega g^t \Omega\}
\]
along with the corresponding Lie algebra
\[
        \frak{sp}_{2n}(\mathbb{F}_q) = \{x \in \frak{gl}_{2n}(\mathbb{F}_q) \mid -x = -\Omega x^t \Omega\}.
\]
Define
\begin{align*}
        USp_{2n}(\mathbb{F}_q) &= UT_{2n}(\mathbb{F}_q) \cap Sp_{2n}(\mathbb{F}_q) \text{ and} \\
        \frak{usp}_{2n}(\mathbb{F}_q) &= \frak{ut}_{2n}(\mathbb{F}_q) \cap \frak{sp}_{2n}(\mathbb{F}_q).
\end{align*}
Define an antiautomorphism $\dagger$ of $\frak{ut}_{2n}(\mathbb{F}_q)$ by $x^\dagger = -\Omega x^t \Omega$. Note that $\dagger$ satisfies the conditions required by Theorem~\ref{sctofu}, and furthermore
\begin{align*}
        USp_{2n}(\mathbb{F}_q) &= \{g \in UT_{2n}(\mathbb{F}_q) \mid g^{-1} = g^\dagger\} \text{ and} \\
        \frak{usp}_{2n}(\mathbb{F}_q) &= \{x \in \frak{ut}_{2n}(\mathbb{F}_q) \mid -x = x^\dagger\}.
\end{align*}
Define $K_u$ and $\chi_\lambda$ as in \ref{eqku} and \ref{eqchilambda} with $U = USp_{2n}(\mathbb{F}_q)$ and $\frak{u} = \frak{usp}_{2n}(\mathbb{F}_q)$. By Theorem~\ref{sctofu}, there is a supercharacter theory of $USp_{2n}(\mathbb{F}_q)$ with superclasses $\{K_u\}$ and supercharacters $\{\chi_\lambda\}$.

\bigbreak

In \cite{MR2537684}, Andr\'e--Neto have also constructed supercharacter theories of $USp_{2n}(\mathbb{F}_q)$. As was the case with the unipotent orthogonal groups, the superclasses are the sets of the form $USp_{2n}(\mathbb{F}_q)\cap K_g$, where $K_g$ is a superclass of $UT_{2n}(\mathbb{F}_q)$ under the algebra group supercharacter theory. In particular, the following theorem follows from Theorem~\ref{superclassintersect}.

\begin{thm} The supercharacter theory of $USp_{2n}(\mathbb{F}_q)$ defined above coincides with that of Andr\'e--Neto in \cite{MR2537684}.
\end{thm}

\bigbreak
We can also construct supercharacter theories of certain subgroups of $USp_{2n}(\mathbb{F}_q)$ just as we did for $UO_n(\mathbb{F}_q)$. The antiautomorphism $\dagger$ as defined above restricts to an antiautomorphism of $U_\mathcal{P}$ for any mirror poset. Furthermore,
\begin{align*}
        USp_{2n}(\mathbb{F}_q)\cap U_\mathcal{P} &= \{g \in U_\mathcal{P} \mid g^{-1} = g^\dagger\} \text{ and} \\
        \frak{usp}_{2n}(\mathbb{F}_q)\cap \frak{u}_\mathcal{P} &= \{x \in \frak{u}_\mathcal{P} \mid -x = x^\dagger\}.
\end{align*}

Define $K_u$ and $\chi_\lambda$ as in \ref{eqku} and \ref{eqchilambda} with $U = USp_{2n}(\mathbb{F}_q)\cap U_\mathcal{P}$ and $\frak{u} = \frak{usp}_{2n}(\mathbb{F}_q)\cap \frak{u}_\mathcal{P}$. By Theorem~\ref{sctofu}, there is a supercharacter theory of $USp_{2n}(\mathbb{F}_q)\cap U_\mathcal{P}$ with superclasses $\{K_u\}$ and supercharacters $\{\chi_\lambda\}$. By Theorem~\ref{superclassintersect}, the superclasses are of the form $K_g \cap USp_{2n}(\mathbb{F}_q)$ where $K_g$ is a superclass of $U_\mathcal{P}$ in the algebra group supercharacter theory. In particular, if $U$ is the unipotent radical of a parabolic subgroup of $Sp_{2n}(\mathbb{F}_q)$ then $U = USp_{2n}(\mathbb{F}_q)\cap U_\mathcal{P}$ for some mirror poset $\mathcal{P}$.

\section{Background}\label{background}

In order to prove Theorem~\ref{sctofu} we need a number of lemmas with regards to the interactions between groups and vector spaces. In this section we will establish these results before applying them in Sections~\ref{alggpsct} and \ref{mainresultproof}.

\subsection{Linear actions of groups on vector spaces}
Let $G$ be a finite group acting linearly on a finite dimensional vector space $V$ over a finite field. There is a corresponding linear action on the dual space $V^{*}$; for $\lambda \in V^{*}$, $g \in G$, and $v \in V$, define
\[
        (g \cdot \lambda) (v) = \lambda (g^{-1} \cdot v).
\]
The following lemma relating the number of orbits of these two actions appears in \cite{MR2373317}.
\begin{lem}[\cite{MR2373317}, Lemma 4.1]\label{orbitcounting} The actions of $G$ on $V$ and $V^{*}$ have the same number of orbits.\end{lem}

\subsection{Complex-valued functions of certain $p$-groups}

Let $G$ be a finite group, and let $V$ be a vector space over the finite field $\mathbb{F}_q$ such that there exists a bijection $f:G \to V$. Let $\theta : \mathbb{F}_q^+ \to \mathbb{C}^\times$ be a nontrivial linear character. We can use the vector space structure of $V$ to study the space of functions from $G$ to $\mathbb{C}$. The following lemma is a consequence of Lemma 5.1 in \cite{MR2373317}.

\begin{lem}\label{orthonormal} Let $G$, $V$ and $\theta$ be as above.
\begin{enumerate}[label=\emph{(\alph*)}]
\item The set of functions $\theta \circ \lambda$, where $\lambda \in V^*$, form an orthonormal basis for the space of functions from $V$ to $\mathbb{C}$.
\item The set of functions $\theta \circ \lambda \circ f$, where $\lambda \in V^*$, form an orthonormal basis for the space of functions from $G$ to $\mathbb{C}$.
\end{enumerate}
\end{lem}

The next lemma will be useful in describing certain induced characters.

\begin{lem}\label{sum0}
    Let $V$ be a vector space of finite dimension over $\mathbb{F}_q$ with subspace $W$, and let $\lambda \in W^*$. Then
\[
        \frac{|W|}{|V|}\sum_{\substack{\mu \in V^* \\ \mu|_W = \lambda}} (\theta \circ \mu) (v) = \left \{ \begin{array}{ll} (\theta \circ \lambda) (v) & \quad \textup{if } v \in W, \\
        0 & \quad \textup{else.} \end{array} \right.
\]
\end{lem}

\begin{proof}
 Let $W'$ be a subspace of $V$ such that $V = W \oplus W'$. Let $v \in V$, and write $v = w +w'$, where $w \in W$ and $w' \in W'$. Then
 \begin{align*}
        \frac{|W|}{|V|}\sum_{\substack{\mu \in V^* \\ \mu|_W = \lambda}} (\theta \circ \mu) (v)
        & = \frac{|W|}{|V|}\sum_{\substack{\mu \in V^* \\ \mu|_W = \lambda}} (\theta \circ \mu) (w+w') \\
        & =  \frac{|W|(\theta \circ \lambda) (w)}{|V|}\sum_{\substack{\mu \in V^* \\ \mu|_W = \lambda}} (\theta \circ \mu)(w').
\end{align*}
Observe that the set of functionals $\mu|_{W'}$ such that $\mu|_W = \lambda$ is exactly $(W')^*$. Furthermore, for $w' \in W'$,
\[
    \sum_{\eta \in (W')^*} (\theta \circ \eta )(w') = \left\{
    \begin{array}{ll}
    |W'| &\quad \text{ if } w'=0, \\
    0 &\quad \text{ else,} \end{array} \right.
\]
as $\theta$ is nontrivial.
\end{proof}

\begin{cor}\label{regrep} If $f(1)=0$, then
\[
        \sum_{\lambda \in V^*} \theta \circ \lambda \circ f
\]
is the regular character of $G$. \end{cor}

The groups we are studying in this paper are all naturally in bijection with a vector space. We can consider an algebra group $G = 1+\frak{g}$ along with the bijection
\begin{align*}
        f: G &\to \frak{g} \\
        1+x & \mapsto x .
\end{align*}
 We can also take our group $U$ to be as defined in Section~\ref{subgroups}, along with the corresponding Lie algebra $\frak{u}$ and a Springer morphism $f:U \to \frak{u}$. In these two cases, we can use the adjoint action of the group on its Lie algebra to understand certain induced representations.

\begin{lem}\label{induced} Suppose that $G$ is a finite group, $V$ is a vector space over $\mathbb{F}_q$, and $f:G \to V$ is a bijection. Suppose further that there is an action
\begin{align*}
        G \times V &\to V \\
        (g,v) & \mapsto g \cdot v.
\end{align*}
such that $f(h gh^{-1}) = h \cdot f(g)$ for all $g,h \in G$. If $H$ is a subgroup of $G$ such that $f(H)=W$ is a subspace of $V$, and $\lambda \in W^*$ is a functional such that $\theta \circ \lambda \circ f$ is a class function of $H$, then
\[
        \textup{Ind}_H^G(\theta \circ \lambda \circ f) =
        \frac{1}{|G|}\sum_{g \in G} \sum_{\substack{\mu \in V^* \\ \mu|_{W} = \lambda}} \theta \circ (g \cdot \mu) \circ f
\]
\end{lem}

\begin{proof} Define $\gamma:G \to \mathbb{C}$ by
\[
        \gamma(g) = \left\{\begin{array}{ll}
        (\theta \circ \lambda \circ f)(g) & \quad \text{if } g \in H, \\
        0 & \quad \text{otherwise.} \end{array}\right.
\]
By Lemma~\ref{sum0},
\[
        \gamma = \frac{|H|}{|G|}\sum_{\substack{\mu \in V^* \\ \mu|_W = \lambda}} \theta \circ \mu \circ f.
\]
For $g \in G$,
\begin{align*}
        \text{Ind}_H^G(\theta \circ \lambda \circ f)(g)
        & = \frac{1}{|H|}\sum_{\substack{h \in G \\ hgh^{-1} \in H}} (\theta \circ \lambda \circ f)(hgh^{-1}) \\
        & = \frac{1}{|H|}\sum_{h \in G} \gamma(hgh^{-1}) \\
        & = \frac{1}{|G|}\sum_{h \in G}\sum_{\substack{\mu \in V^* \\ \mu|_W = \lambda}} (\theta \circ \mu \circ f)(hgh^{-1}) \\
        & = \frac{1}{|G|}\sum_{h \in G}\sum_{\substack{\mu \in V^* \\ \mu|_W = \lambda}} (\theta \circ (h^{-1} \cdot\mu) \circ f)(g) \\
        & = \frac{1}{|G|}\sum_{h \in G}\sum_{\substack{\mu \in V^* \\ \mu|_W = \lambda}} (\theta \circ (h \cdot \mu) \circ f)(g),
\end{align*}
using the fact that $f(h gh^{-1}) = h \cdot f(g)$.
\end{proof}

\section{Supercharacter theories of algebra groups}\label{alggpsct}

Let $G = 1+\frak{g}$ be an algebra group over the field $\mathbb{F}_q$, where $q$ is a power of a prime. Diaconis--Isaacs construct a supercharacter theory of $G$ in \cite{MR2373317}, which we describe here. Define
\begin{align*}
        f: G &\to \frak{g} \\
        g & \mapsto g-1 .
\end{align*}
Note that $G$ acts by left multiplication, right multiplication, and conjugation on $\frak{g}$ (with corresponding actions on $\frak{g}^*$). For $g \in G$, define
\[
        K_g = \{h \in G \mid f(h) \in Gf(g)G\}.
\]

\bigbreak

Let $\theta : \mathbb{F}_q^+ \to \mathbb{C}^\times$ be a nontrivial homomorphism. For $\lambda \in \frak{g}^*$, define
\[
        \chi_\lambda = \frac{|G\lambda|}{|G\lambda G|} \sum_{\mu \in G \lambda G}
        \theta \circ \mu \circ f.
\]
Note that the set $G \lambda G$ is the orbit of $\lambda$ under the action of $G \times G$ on $\frak{g}^*$ defined by $((g,h)\cdot\lambda)(x)=\lambda(g^{-1}xh)$. In particular, $G\lambda$ is the orbit of $\lambda$ under the action of the normal subgroup $G \times \{1\}$. It follows that $|G\mu|=|G\lambda|$ for all $\mu \in G\lambda G$, and the definition of $\chi_\lambda$ is independent of the choice of representative of $G\lambda G$.

\begin{thm}[\cite{MR2373317}]\label{sctalggp} Let $K_g$ and $\chi_\lambda$ be as above.
\begin{enumerate}[label=\emph{(\alph*)}]
\item The functions $\chi_\lambda$ are characters of $G$.
\item The partition of $G$ given by $\mathcal{K} = \{K_g \mid g \in G\}$, along with the set of characters $\{\chi_\lambda \mid \lambda \in \frak{g}^*\}$, form a supercharacter theory of $G$.
\end{enumerate}
\end{thm}

We present a proof of this result as motivation for our proof of Theorem~\ref{sctofu}; our method is different from that in \cite{MR2373317}, although many of the ideas are similar. We will prove (a) by proving a more specific result given in Theorem~\ref{chara}. Assuming (a), we have the following.

\begin{proof}[Proof of \emph{(b)}.]

We need to show that conditions (1)-(3) in the definition of a supercharacter theory (see Section~\ref{sctig}) are satisfied. For (1), note that $|\mathcal{K}|$ is the number of orbits of the action of $G \times G$ on $\frak{g}$ defined by $(g,h) \cdot x = gxh^{-1}$. At the same time, $|\{\chi_\lambda \mid \lambda \in \frak{g}^*\}|$ is the number of orbits of the corresponding action of $G \times G$ on $\frak{g}^*$. By Lemma~\ref{orbitcounting}, the number of orbits of the two actions are equal.

\bigbreak

To demonstrate that condition (2) holds, choose $g \in G$ and $\lambda \in \frak{g}^*$; we have that
\begin{align*}
        \chi_\lambda(g)
        & = \frac{|G\lambda|}{|G\lambda G|} \sum_{\mu \in G \lambda G}
        (\theta \circ \mu \circ f)(g) \\
        & = \frac{|G\lambda|}{|G|^2} \sum_{h,k \in G}
        (\theta \circ h \lambda k \circ f)(g) \\
        & = \frac{|G\lambda|}{|G|^2} \sum_{h,k \in G}
        (\theta \circ \lambda)(h^{-1}f(g)k^{-1}) \\
        & = \frac{|G\lambda|}{|K_g|} \sum_{h \in K_g}
        (\theta \circ \lambda)(f(h)).
\end{align*}
It follows that $\chi_\lambda(g)$ only depends on the superclass of $g$.

\bigbreak

Condition (3) follows from Lemma~\ref{orthonormal} and Corollary~\ref{regrep}.
\end{proof}

It remains to prove (a). Define
\[
        \frak{l}_\lambda = \{x \in \frak{g} \mid \lambda(yx) = 0 \text{ for all } y \in \frak{g}\},
\]
and let $L_\lambda = 1+\frak{l}_\lambda$. It is worth mentioning that our notation differs from that of Diaconis--Isaacs. We define $\frak{l}_\lambda$ as above so that $\frak{l}_\lambda$ is the left ideal of $\frak{g}$ that corresponds to the left orbit $G\lambda$ as follows.

\begin{lem}[Lemma 4.2 (d), \cite{MR2373317}]\label{glambda} With notation as above, $G\lambda = \{\mu \in \frak{g}^* \mid \mu|_{\frak{l}_\lambda} = \lambda|_{\frak{l}_\lambda}\}$.
\end{lem}

Diaconis--Isaacs prove the following result as part of Theorem 5.4 in \cite{MR2373317}.

\begin{lem} The function $\textup{Res}_{L_\lambda}^G(\theta \circ \lambda \circ f)$ is a linear character of $L_\lambda$.
\end{lem}

\begin{proof} Let $x,y \in \frak{l}_\lambda$; then
\begin{align*}
       (\theta \circ \lambda \circ f)((1+x)(1+y))
       & = \theta(\lambda(x+y+xy)) \\
       & = \theta(\lambda(x+y)) \\
       & = (\theta \circ \lambda \circ f)(1+x)(\theta \circ \lambda \circ f)(1+y).
\end{align*}
\end{proof}

We can now prove that the functions $\chi_\lambda$ are characters of $G$.

\begin{thm}[\cite{MR2373317}]\label{chara} With $\chi_\lambda$ as defined above,
\[
        \chi_\lambda = \textup{Ind}_{L_\lambda}^G (\textup{Res}_{L_\lambda}^G(\theta \circ \lambda \circ f))
\]
\end{thm}

\begin{proof} By Lemma~\ref{induced} and Lemma~\ref{glambda},
\begin{align*}
        \textup{Ind}_{L_\lambda}^G (\textup{Res}_{L_\lambda}^G(\theta \circ \lambda \circ f))
        & = \frac{1}{|G|}\sum_{g \in G} \sum_{\substack{\mu \in \frak{g}^* \\ \mu|_{\frak{l}_\lambda} = \lambda|_{\frak{l}_\lambda}}} \theta \circ g\mu g^{-1} \circ f \\
        & = \frac{1}{|G|}\sum_{g \in G} \sum_{\mu \in G\lambda} \theta \circ g\mu g^{-1} \circ f \\
        & = \frac{|G\lambda|}{|G|^2}\sum_{g \in G} \sum_{h \in G} \theta \circ g (h\mu)g^{-1} \circ f \\
        & = \frac{|G\lambda|}{|G|^2}\sum_{h,g \in G}\theta \circ g\mu h \circ f \\
        & = \frac{|G\lambda|}{|G \lambda G|}\sum_{\mu \in G \lambda G} \theta \circ \mu  \circ f.
\end{align*}
\end{proof}

\section{Supercharacter theories of unipotent groups defined by anti-involutions}\label{mainresultproof}

In this section we construct supercharacter theories of the groups $U$ that were introduced in Section~\ref{mainresult}. Let $q$ be a power of an odd prime, and let $G = 1+\frak{g}$ be a pattern subgroup of $UT_n(\mathbb{F}_{q^k})$ for some $n$ and $k$. We consider $\frak{g}$ as an $\mathbb{F}_q$-algebra; let $\dagger$ be an anti-involution of $\frak{g}$ such that $(\alpha e_{ij})^\dagger \in \mathbb{F}_{q^k}^\times e_{\bar{j}\bar{i}}$ for all $\alpha \in \mathbb{F}_{q^k}^\times$ (recall that $\bar{i}=n+1-i$). Define
\[
        U = \{u \in G \mid u^{\dagger} = u^{-1}\}
\]
and
\[
        \frak{u} = \{x \in \frak{g} \mid x^\dagger = -x\}.
\]
Let $f$ be any Springer morphism and let $\theta:\mathbb{F}_q^+ \to \mathbb{C}^\times$ be a nontrivial homomorphism. Recall that there are left actions of $G$ and $\frak{g}$ on $\frak{g}$ defined by
\begin{align*}
    g \cdot x & = gxg^\dagger \\
    y * x & = yx + xy^\dagger
\end{align*}
for $g \in G$ and $x,y \in \frak{g}$, along with corresponding actions on $\frak{g}^*$.

\bigbreak

In the construction of the supercharacter theories of algebra groups, the normal subgroup $G \times 1$ of $G \times G$ plays an important role. We need an analogous subgroup of $G$ to construct a supercharacter theory of $U$. Let
\[
\frak{h} = \bigg\{ x\in \frak{g} \:\bigg|\: x_{ij} = 0 \text{ if } j\leq\frac{n}{2}\bigg\},
\]
and define $H = \frak{h}+1$. Note that $\frak{h}$ is a two-sided ideal of $\frak{g}$, hence $H$ is a normal subgroup of $G$.

\bigbreak

For $u \in U$ and $\lambda \in \frak{u}^*$, define
\[
        K_u = \{ v \in U \mid f(v) \in G \cdot f(u)\}
\]
and
\[
        \chi_\lambda = \frac{|H \cdot \lambda|}{|G \cdot \lambda|} \sum_{\mu \in G \cdot \lambda}
        \theta \circ \mu \circ f.
\]
As $H$ is a normal subgroup of $G$, $|H \cdot \lambda|$ is independent of the choice of orbit representative of $G \cdot \lambda$.

\begin{thm}\label{sctofu} Let $K_u$ and $\chi_\lambda$ be as above.
\begin{enumerate}[label=\emph{(\alph*)}]
\item The functions $\chi_\lambda$ are characters of $U$.
\item The partition of $U$ given by $\mathcal{K} = \{K_u \mid u \in U\}$, along with the set of characters $\{\chi_\lambda \mid \lambda \in \frak{u}^*\}$, form a supercharacter theory of $U$.
\end{enumerate}
\end{thm}

We will prove (a) by proving a more specific result given by Theorem~\ref{ucharacters}. Assuming (a), we have the following.

\begin{proof}[Proof of \emph{(b)}] We need to show that conditions (1)-(3) in the definition of a supercharacter theory (see Section~\ref{sctig}) are satisfied. For condition (1), note that $|\mathcal{K}|$ is the number of orbits of the action of $G$ on $\frak{u}$. At the same time, $|\{\chi_\lambda \mid \lambda \in \frak{u}^*\}|$ is the number of orbits of the corresponding action of $U$ on $\frak{u}^*$. By Lemma~\ref{orbitcounting}, the number of orbits of the two actions are equal.

\bigbreak

To demonstrate that condition (2) holds, choose $u \in U$ and $\lambda \in \frak{u}^*$; we have that
\begin{align*}
        \chi_\lambda(u)
        & = \frac{|H \cdot \lambda|}{|G\cdot \lambda|} \sum_{\mu \in G\cdot \lambda}
        (\theta \circ \mu \circ f)(u) \\
        & = \frac{|H \cdot \lambda|}{|G|} \sum_{g \in G}
        (\theta \circ g\cdot \lambda  \circ f)(u) \\
        & = \frac{|H \cdot \lambda|}{|G|} \sum_{g \in G}
        (\theta \circ \lambda)(g^{-1} \cdot f(u)) \\
        & = \frac{|H \cdot \lambda|}{|K_u|} \sum_{v \in K_u}
        (\theta \circ \lambda)(f(v)).
\end{align*}
It follows that $\chi_\lambda(u)$ only depends on the superclass of $u$.

\bigbreak

Condition (3) follows from Lemma~\ref{orthonormal} and Corollary~\ref{regrep}.
\end{proof}

It remains to prove (a). For a fixed $\lambda \in \frak{g}^*$, we define several subalgebras of $\frak{g}$. Let
\begin{align*}
        \frak{l}_\lambda &= \{x \in \frak{g} \mid \lambda(yx)=0 \text{ for all }y \in \frak{h}\}, \\
        \frak{r}_\lambda &= \{x \in \frak{g} \mid \lambda(xy)=0 \text{ for all }y \in \frak{h}^\dagger\}, \text{ and } \\
        \frak{g}_\lambda &= \frak{l}_\lambda \cap \frak{r}_\lambda.
\end{align*}
We also define the corresponding algebra subgroups
\begin{align*}
        L_\lambda &= 1+ \frak{l}_\lambda, \\
        R_\lambda &= 1+ \frak{r}_\lambda, \text{ and} \\
        G_\lambda &= 1 + \frak{g}_\lambda = L_\lambda \cap R_\lambda.
\end{align*}

\begin{lem}\label{hlambda} With notation as above, $H\lambda = \{\mu \in \frak{g}^* \mid \mu|_{\frak{l}_\lambda} = \lambda |_{\frak{l}_\lambda}\}$. \end{lem}

\begin{proof} Note that $H\lambda -\lambda$ is a subspace of $\frak{g}^*$, and for $x \in \frak{g}$ and $y \in \frak{h}$,
\[
        ((1+y)^{-1}\lambda-\lambda)(x) = \lambda(yx).
\]
It follows that $\frak{l}_\lambda = \{x \in \frak{g} \mid \mu(x) = 0 \text{ for all } \mu \in H\lambda - \lambda\}$, hence $H\lambda - \lambda = \{\mu \in \frak{g}^* \mid \mu(x)=0 \text{ for all } x \in \frak{l}_\lambda\}$.
\end{proof}

\begin{lem}\label{kernel} For any $\lambda \in \frak{g}^*$, we have that $\lambda(xy) = 0$ for all $x,y \in \frak{g}_\lambda$.
\end{lem}

\begin{proof} For $x,y \in \frak{g}_\lambda$, define elements $x'$ and $y'$ of $\frak{g}$ by
\[
        (x')_{ij} = \left\{\begin{array}{ll}
        x_{ij} & \quad \text{if } j>\frac{n}{2} \\
        0 &\quad \text{else} \end{array}\right.
\]
and
\[
        (y')_{ij} = \left\{\begin{array}{ll}
        y_{ij} & \quad \text{if } i\leq \frac{n}{2} \\
        0 &\quad \text{else.}\end{array}\right.
\]
Note that
\[
        (x'y)_{ij} = \sum_{k>\frac{n}{2}} x_{ik}y_{kj}
\]
and
\[
        (xy')_{ij} = \sum_{k\leq\frac{n}{2}} x_{ik}y_{kj} .
\]
It follows that $xy = x'y+xy'$. Observe that $x' \in \frak{h}$ and $y' \in \frak{h}^\dagger$; as $x,y \in \frak{g}_\lambda$,
\[
        \lambda(xy) = \lambda(x'y)+\lambda(xy') = 0.
\]
\end{proof}

A corollary of this result will allow us to conclude that our supercharacter theories are independent of the choice of Springer morphism $f$.

\begin{cor}\label{linchar} Let $\lambda \in \frak{g}^*$; then
\begin{enumerate}[label=\emph{(\alph*)}]
\item the function $\textup{Res}_{G_\lambda}^G(\theta \circ \lambda \circ f)$ is a linear character of $G_\lambda$, and
\item if $f'$ is another Springer morphism, $\textup{Res}_{G_\lambda}^G(\theta \circ \lambda \circ f) = \textup{Res}_{G_\lambda}^G(\theta \circ \lambda \circ f')$.
\end{enumerate}
\end{cor}

\begin{proof} Let $x,y \in \frak{g}_\lambda$; then
\[
        f((1+x)(1+y)) = x+y + p(x,y),
\]
where $p(x,y)$ is a polynomial in $x$ and $y$ with all terms of degree at least two. By Lemma~\ref{kernel}, $\lambda(p(x,y))=0$. It follows that
\begin{align*}
        (\theta \circ \lambda \circ f)((1+x)(1+y))
        & = \theta(\lambda(x+y)) \\
        & = \theta(\lambda(x))\theta(\lambda(y)).
\end{align*}
At the same time,
\begin{align*}
        (\theta \circ \lambda \circ f)(1+x)(\theta \circ \lambda \circ f)(1+y)
        & = \theta\bigg(\lambda\bigg(x+\sum_{i=2}^\infty a_i x^i\bigg)\bigg)\theta\bigg(\lambda\bigg(y+\sum_{i=2}^\infty a_i y^i\bigg)\bigg) \\
        & = \theta(\lambda(x))\theta(\lambda(y)),
\end{align*}
and $\textup{Res}_{G_\lambda}^G(\theta \circ \lambda \circ f)$ is a linear character of $G_\lambda$. Note that
\[
        \textup{Res}_{G_\lambda}^G(\theta \circ \lambda \circ f)(x) = \theta (\lambda(x)),
\]
a formula independent of $f$, proving (b).
\end{proof}

\bigbreak

There are two properties of the subgroup $H$ that will be useful in future calculations.

\begin{lem}\label{actionscommute} For $h \in H$ and $x \in \frak{g}$, $h \cdot x = (h-1)*x+x$.
\end{lem}

\begin{proof} Let $h = 1+y$; then
\[
        h \cdot x = yxy^\dagger+yx+xy^\dagger +x
\]
and
\[
        (h-1)*x+x = yx+xy^\dagger+x.
\]
It suffices to show that $\frak{h}\frak{g}\frak{h}^\dagger = 0$. Note that $\frak{h}\frak{g}\frak{h}^\dagger$ is generated by elements of the form $e_{ij}e_{kl}e_{rs}$ with $j>\frac{n}{2}$ and $r < \frac{n}{2}+1$. This means that $j \geq r$, and as $k < l$,  $e_{ij}e_{kl}e_{rs}=0$.
\end{proof}

\begin{lem} We have that $G = HU$.
\end{lem}

\begin{proof} Let $x \in \frak{g}$; define $y \in \frak{u}$ by
\[
        y_{ij} = \left\{\begin{array}{ll}
        x_{ij} &\quad \text{if } j\leq \frac{n}{2} \\
        -(x^\dagger)_{ij} & \quad \text{if } i \geq \frac{n}{2}+1 \\
        0 &\quad \text{else}\end{array}\right. .
\]
Note that $x-y \in \frak{h}$, hence $\frak{g} = \frak{h}+\frak{u}$. It follows that
\[
        |HU|  = \frac{|H||U|}{|H \cap U|}
         = \frac{|f(H)||f(U)|}{|f(H \cap U)|}
         = \frac{|\frak{h}||\frak{u}|}{|\frak{h} \cap \frak{u}|}
         = |\frak{g}| = |G|
\]
and the result follows.
\end{proof}

In order to use the above results to study $U$, we will need to extend the elements of $\frak{u}^*$ to elements of $\frak{g}^*$. There are of course many possible ways to do this, however there is one natural choice in our case.

\begin{lem} Given any $\lambda \in \frak{u}^*$, there exists a unique $\eta \in \frak{g}^*$ such that $\eta|_\frak{u} = \lambda$ and $\eta(x) = -\eta(x^\dagger)$ for all $x \in \frak{g}$.
\end{lem}

\begin{proof} For $x \in \frak{g}$, let $\eta(x) = \frac{1}{2}\lambda(x-x^\dagger)$. This definition makes sense as $x-x^\dagger \in \frak{u}$ for all $x \in \frak{g}$. Note that $\eta \in \frak{g}^*$ and $\eta(x) = \lambda(x)$ for all $x \in \frak{u}$. It is also clear that $\eta(x) = -\eta(x^\dagger)$ for all $x \in \frak{g}$.

\bigbreak

The uniqueness of $\eta$ follows from the fact that $\mu(x-x^\dagger) = 2\mu(x)$ for any $\mu$ satisfying $\mu(x) = -\mu(x^\dagger)$. This means that such $\mu$ is determined only by its values on $\frak{u}$, hence $\eta$ is unique.
\end{proof}

\begin{lem}\label{setequal}
    Let $\lambda \in \frak{u}^*$, and let $\eta \in \frak{g}^*$
be the extension of $\lambda$ described above. Then the sets $\{\mu|_\frak{u} \mid \mu \in H\eta \}$ and $H \cdot \lambda$ are equal.
\end{lem}

\begin{proof}
Let $h \in H$ and $x \in \frak{u}$; then
\begin{align*}
        h^{-1} \cdot \lambda (x) & = \lambda (h \cdot x) \\
        & = \eta ((h-1) * x +x)
\end{align*}
by Lemma~\ref{actionscommute}. Note that
\begin{align*}
        \eta ((h-1) * x +x)
        & = \eta ((h-1)x+x(h-1)^\dagger) + \eta(x) \\
        & = 2\eta((h-1)x) + \eta(x)
\end{align*}
by the fact that $\eta(y) = -\eta(y^\dagger)$ for all $y \in \frak{g}$. Finally,
\begin{align*}
        2\eta((h-1)x) + \eta(x)
        & = \eta((2(h-1)+1)x) \\
        & = (2h-1)^{-1}\eta(x).
\end{align*}
The claim follows from the fact that the map $h \mapsto (2h-1)$ is a bijection from $H$ to $H$.
\end{proof}

We are now ready to prove that for $\lambda \in \frak{u}^*$ the function
\[
        \chi_\lambda = \frac{|H \cdot \lambda|}{|G \cdot \lambda|}\sum_{\mu \in G \cdot \lambda} \theta \circ \mu \circ f
\]
is a character of $U$. Let $\eta$ be the element of $\frak{g}^*$ associated to $\lambda$ as above, and define $U_\lambda = U \cap G_\eta$ and $\frak{u}_\lambda = f(U_\lambda)$.

\begin{thm}\label{ucharacters} We have that
\[
       \chi_\lambda = \textup{Ind}_{U_\lambda}^U(\textup{Res}_{U_\lambda}^U (\theta \circ \lambda \circ f)).
\]

\end{thm}

\begin{proof} Note that $\text{Res}_{U_\lambda}^U (\theta \circ \lambda \circ f) = \text{Res}_{U_\lambda}^G (\theta \circ \eta \circ f)$; as $U_\lambda \subseteq G_\eta$, by Corollary~\ref{linchar} the function $\text{Res}_{U_\lambda}^U (\theta \circ \lambda \circ f)$ is a linear character of $U_\lambda$.

\bigbreak

It is clear that $\frak{u}_\lambda = \frak{u} \cap \frak{g}_\eta$, but in fact $\frak{u}_\lambda = \frak{u} \cap \frak{l}_\eta$. This is a consequence of the fact that if $x \in \frak{u}$ and $x \in \frak{l}_\eta$, then $x \in \frak{r}_\eta$. It follows that
\begin{align*}
        \{\mu \in \frak{u}^* \mid \mu(x) = \lambda(x) \text{ for all } x \in \frak{u}_\lambda\}
        & = \{\kappa|_\frak{u} \mid \kappa(x) = \eta(x) \text{ for all } x \in \frak{l}_\eta\} \\
        & = \{\kappa|_\frak{u} \mid \kappa \in H\eta\} \\
        & = H \cdot \lambda
\end{align*}
by Lemma~\ref{hlambda} and Lemma~\ref{setequal}. By Lemma~\ref{induced},
\begin{align*}
        \textup{Ind}_{U_\lambda}^U(\textup{Res}_{U_\lambda}^U (\theta \circ \lambda \circ f))
        & = \frac{1}{|U|}\sum_{u \in U} \sum_{\substack{\mu \in \frak{u}^* \\ \mu|_{\frak{u}_\lambda} = \lambda|_{\frak{u}_\lambda}}} \theta \circ u \mu u^{-1} \circ f \\
        & = \frac{1}{|U|}\sum_{u \in U} \sum_{\mu \in H \cdot \lambda} \theta \circ u \mu u^{-1} \circ f \\
        & = \frac{|H \cdot \lambda|}{|H||U|}\sum_{u \in U} \sum_{h \in H} \theta \circ u (h \cdot \lambda) u^{-1} \circ f.
\end{align*}
Recall that $u \cdot x = u xu^{-1}$ for all $x \in \frak{g}$ and $HU = G$. It follows that
\begin{align*}
        \frac{|H \cdot \lambda|}{|H||U|}\sum_{u \in U} \sum_{h \in H} \theta \circ u (h \cdot \lambda)u^{-1} \circ f
        & = \frac{|H \cdot \lambda|}{|H||U|}\sum_{u \in U} \sum_{h \in H} \theta \circ  (uh \cdot \lambda) \circ f \\
        & = \frac{|H \cdot \lambda|}{|G|}\sum_{g \in G} \theta \circ  (g \cdot \lambda) \circ f \\
        & = \frac{|H \cdot \lambda|}{|G \cdot \lambda|}\sum_{\mu \in G \cdot \lambda} \theta \circ  \mu \circ f
\end{align*}
which is by definition $\chi_\lambda$.
\end{proof}
It is worth noting that the function $\text{Res}_{U_\lambda}^U(\theta \circ \lambda \circ f)$ is independent of the choice of Springer morphism $f$, and as such the $\chi_\lambda$ do not depend on $f$. We also have a connection between the supercharacter theory of $U$ and the supercharacter theory of the algebra group $G$.

\begin{thm}\label{superclassintersect}
    The superclasses of $U$ are exactly the sets of the form $U \cap K_g$, where $K_g$ is some superclass of $G$.
\end{thm}

\begin{proof} Note that for each $u \in U$ there exists $g \in G$ such that $f(u) = (u-1)g$. It follows that each superclass of $U$ is contained in some superclass of $G$. We want to show that each superclass of $G$ contains at most one superclass of $U$.

\bigbreak

Note that for $g,h \in G$ and $x \in \frak{g}$, $gxh = h^\dagger(h^{-\dagger}gx)h$. As such, it suffices to show that if $x \in \frak{u}$ and $gx \in \frak{u}$ for some $g \in G$, then $gx = hxh^\dagger$ for some $h \in G$.

\bigbreak

Assume that $x,gx \in \frak{u}$; then
\[
    gx  = -(gx)^\dagger
     = -x^\dagger g^\dagger
     = xg^\dagger.
\]
Let $k$ be an odd integer such that $g^{2k} = g$ (such $k$ must exist as $g$ has odd order). Then
\[
    gx  = g^{2k}x
     = g^k x (g^k)^\dagger;
\]
let $h = g^k$.
\end{proof}

In Section 3 of \cite{MR2264135}, Andr\'e--Neto show that their supercharacter theories of the unipotent orthogonal and symplectic groups have superclasses of the form $U \cap K_g$ as well. This demonstrates that our supercharacter theory coincides with theirs if $U = UO_n(\mathbb{F}_q)$ or $USp_n(\mathbb{F}_q)$.

\section{Supercharacter theories of the unipotent unitary groups}\label{unitary}

Let $q$ be a power of an odd prime, and for $x \in \frak{gl}_n(\mathbb{F}_{q^2})$, define $\overline{x}$ by $(\overline{x})_{ij} = (x_{ij})^q$. Let
\begin{align*}
        U_n(\mathbb{F}_{q^2}) &= \{g \in GL_n(\mathbb{F}_{q^2}) \mid g^{-1} = J\overline{g}^tJ \} \text{ and} \\
        \frak{u}_n(\mathbb{F}_{q^2}) &= \{x \in \frak{gl}_n(\mathbb{F}_{q^2}) \mid -x = J\overline{x}^tJ \},
\end{align*}
and let
\begin{align*}
        UU_n(\mathbb{F}_{q^2}) &= U_n(\mathbb{F}_{q^2}) \cap UT_n(\mathbb{F}_{q^2}) \text{ and} \\
        \frak{uu}_n(\mathbb{F}_{q^2}) &= \frak{u}_n(\mathbb{F}_{q^2}) \cap \frak{ut}_n(\mathbb{F}_{q^2}).
\end{align*}
The group $U_n(\mathbb{F}_{q^2})$ is the group of unitary $n \times n$ matrices over $\mathbb{F}_{q^2}$. In this section we construct a supercharacter theory of $UU_n(\mathbb{F}_{q^2})$ using the results from Section~\ref{mainresultproof} and calculate the values of the supercharacters on the superclasses.

\subsection{Construction}

The map $x^\dagger = J\overline{x}^tJ$ defines an antiautomorphism of $\frak{ut}_n(\mathbb{F}_{q^2})$ if we consider $\frak{ut}_n(\mathbb{F}_{q^2})$ as an $\mathbb{F}_q$-algebra. This involution satisfies the conditions required by Theorem~\ref{sctofu}, and futhermore
\begin{align*}
        UU_n(\mathbb{F}_{q^2}) &= \{g \in UT_n(\mathbb{F}_{q^2}) \mid g^{-1} = g^\dagger \} \text{ and} \\
        \frak{uu}_n(\mathbb{F}_{q^2}) &= \{x \in \frak{ut}_n(\mathbb{F}_{q^2}) \mid -x = x^\dagger \}.
\end{align*}

Define $K_u$ and $\chi_\lambda$ as in \ref{eqku} and \ref{eqchilambda} with $U = UU_n(\mathbb{F}_{q^2})$ and $\frak{u} = \frak{uu}_n(\mathbb{F}_{q^2})$. By Theorem~\ref{sctofu}, there is a supercharacter theory of $UU_n(\mathbb{F}_{q^2})$ with superclasses $\{K_u\}$ and supercharacters $\{\chi_\lambda\}$.

\bigbreak

As with the orthogonal and symplectic cases, by Theorem~\ref{superclassintersect} the superclasses are of the form $K_g \cap UU_n(\mathbb{F}_{q^2})$ where $K_g$ is a superclass of $UT_n(\mathbb{F}_{q^2})$ under the algebra group supercharacter theory.

\bigbreak

We can once again construct supercharacter theories of certain subgroups of $UU_n(\mathbb{F}_{q^2})$. The antiautomorphism $\dagger$ as defined above restricts to an antiautomorphism of $U_\mathcal{P}$ for any mirror poset. Furthermore,
\begin{align*}
        UU_{n}(\mathbb{F}_{q^2})\cap U_\mathcal{P} &= \{g \in U_\mathcal{P} \mid g^{-1} = g^\dagger\} \text{ and} \\
        \frak{uu}_{n}(\mathbb{F}_{q^2})\cap \frak{u}_\mathcal{P} &= \{x \in \frak{u}_\mathcal{P} \mid -x = x^\dagger\}.
\end{align*}

Define $K_u$ and $\chi_\lambda$ as in \ref{eqku} and \ref{eqchilambda} with $U = UU_n(\mathbb{F}_{q^2})\cap U_\mathcal{P}$ and $\frak{u} = \frak{uu}_n(\mathbb{F}_{q^2})\cap \frak{u}_\mathcal{P}$. By Theorem~\ref{sctofu}, there is a supercharacter theory of $UU_n(\mathbb{F}_{q^2})\cap U_\mathcal{P}$ with superclasses $\{K_u\}$ and supercharacters $\{\chi_\lambda\}$. By Theorem~\ref{superclassintersect}, the superclasses are of the form $K_g \cap UU_n(\mathbb{F}_{q^2})$ where $K_g$ is a superclass of $U_\mathcal{P}$ in the algebra group supercharacter theory.

\subsection{Superclasses and supercharacters}

In this section we describe the superclasses and supercharacters of $U=UU_n(\mathbb{F}_{q^2})$ in terms of labeled set partitions. Recall that, for $1 \leq i \leq n$, we define $\bar{i} = n+1-i$. A \textit{twisted} $\mathbb{F}_{q}$\textit{-set partition} will refer to an $\mathbb{F}_{q^2}$-set partition $\eta$ of $[n]$ such that if $i \overset{a}{\frown} j \in \eta$ then $\bar{j} \overset{-a^q}{\frown} \bar{i} \in \eta$. In particular, if $i \overset{a}{\frown} \bar{i} \in \eta$, then $a$ satisfies $a^q+a = 0$. For more on labeled set partitions, see \cite{MR2880659}.

\begin{lem}\label{lemscr} Each superclass of $U$ contains exactly one element $u$ with the property that $f(u)$ has at most one nonzero entry in each row and column.
\end{lem}

\begin{proof} The superclasses of $UT_n(\mathbb{F}_{q^2})$ contain exactly one element $u$ such that $f(u)$ has at most one nonzero entry in each row and column. It follows that the superclasses of $U$ contain at most one element with this property. Let $x \in \frak{u}$; we want to row-reduce $x$ using the action of $UT_n(\mathbb{F}_{q^2})$.

\bigbreak

Let $(i,j)$ be such that
\begin{enumerate}
\item $x_{ij} \neq 0$,
\item there exists $k<i$ with $x_{kj} \neq 0$, and
\item there is no other pair $(l,m)$ satisfying properties (1) and (2) with $l\geq i$ and $m\leq j$.
\end{enumerate}

If no such $(i,j)$ exists, then $x$ has at most one nonzero entry in each row and column. Assume that such a pair $(i,j)$ exists. If $k \neq \bar{j}$, we consider
\[
        y = \bigg(1-\frac{x_{kj}}{x_{ij}}e_{ki}\bigg)\cdot x.
\]
If $k = \bar{j}$, we consider
\[
        y = \bigg(1-\frac{x_{kj}}{x_{ij}+x_{ij}^q}e_{ki}\bigg)\cdot x.
\]
The element $y$ is in the same superclass as $x$, but has $y_{kj}=0$. Repeated application of this process will yield an element with at most one nonzero entry in each row and column.
\end{proof}

To each twisted $\mathbb{F}_q$-set partition $\eta$ we assign the element $x_\eta \in \frak{u}$ defined by
\[
        (x_\eta)_{ij} = \left\{\begin{array}{ll}
        a & \quad \text{if } i \overset{a}{\frown} j \in \eta \\
        0 & \text{else.}\end{array}\right.
\]
and the element $u_\eta \in U$ such that $f(u_\eta) = x_\eta$. Note that $x_\eta$ is in fact an element of $\frak{u}$ and has at most one entry in each nonzero row and column.

\begin{cor}\label{sclindexing} The elements
\[
        \{u_\eta \mid \eta \text{ is a twisted }\mathbb{F}_q\text{-partition}\}
\]
are a set of superclass representatives.
\end{cor}

\begin{proof} As mentioned above, $x_\eta$ is an element of $\frak{u}$ and has at most one entry in each nonzero row and column. Conversely, given $x \in \frak{u}$ with at most one entry in each nonzero row and column, define $\eta = \{i \overset{x_{ij}}{\frown} j \mid x_{ij} \neq 0\}$. It is apparent that $x = x_\eta$.
\end{proof}

As there are equal numbers of superclasses and supercharacters, the supercharacters can also be parametrized by twisted $\mathbb{F}_q$-set partitions. Given a twisted $\mathbb{F}_q$-set partition, define $\lambda_\eta \in \frak{u}^*$ by
\[
        \lambda_\eta(x) = \sum_{i \overset{a}{\frown} j \in \eta} ax_{ij}.
\]
\begin{lem}\label{lemor} The set
\[
        \{\lambda_\eta \mid \eta \text{ is a twisted }\mathbb{F}_q\text{-partition}\}
\]
is a set of orbit representatives for the action of $UT_n(\mathbb{F}_{q^2})$ on $\frak{u}^*$.
\end{lem}

The proof of this lemma is similar to that of Lemma~\ref{lemscr}. For a twisted $\mathbb{F}_q$-set partition, we define $\chi^\eta = \chi_{\lambda_\eta}$.

\begin{cor} The superclasses and supercharacters are given by
\[
        \{K_{u_\eta} \mid  \eta \text{ is a twisted }\mathbb{F}_q\text{-partition}\}
\]
and
\[
        \{\chi^\eta \mid \eta \text{ is a twisted }\mathbb{F}_q\text{-partition}\}.
\]
\end{cor}

\subsection{Supercharacter values on superclasses}

The goal of this section is to calculate $\chi^\eta(u_\nu)$, where $\eta$ and $\nu$ are twisted $\mathbb{F}_q$-set partitions. We will call a supercharacter \emph{elementary} if it corresponds to a twisted $\mathbb{F}_q$-set partition of the form $\eta = \{i\overset{a}{\frown}j \cup \bar{j} \overset{-a^q}{\frown}\bar{i}\}$ with $i \neq \bar{j}$ or of the form $\eta = \{i\overset{a}{\frown}\bar{i}\}$ with $a^q+a=0$. In order to simplify calculations, we will show that every supercharacter can be written as a product of distinct elementary supercharacters. This is analogous to the method used in types $A$, $B$, $C$ and $D$ (see \cite{MR1896026,MR2537684}).

\bigbreak

Recall that, for $\lambda \in \frak{u}^*$, the supercharacter $\chi_\lambda$ is induced from a linear character of the subgroup $U_\lambda$ (see Section~\ref{mainresultproof} for specifics). The subgroup $U_\lambda$ is associated to a subalgebra $\frak{u}_\lambda$ of $\frak{u}$. We can describe this subalgebra in terms of the twisted $\mathbb{F}_q$-set partition associated to $\chi_\lambda$.

\begin{lem}\label{a2ulambda} Let $\eta$ be a twisted $\mathbb{F}_q$-set partition; then
\[
        \frak{u}_{\lambda_\eta} = \bigg\{x \in \frak{u} \mid x_{ij} = 0 \textup{ if } i\frown k \in \eta \textup{ with } j<k \textup{ and } j \leq\frac{n+1}{2}\bigg\}.
\]
\end{lem}

\begin{proof} Recall that $\frak{u}_{\lambda_\eta} = \frak{l}_\mu \cap \frak{u}$, where $\mu \in \frak{g}^*$ is the functional defined by $\mu(x) = \frac{1}{2}\lambda_\eta(x-x^\dagger)$. From the definition of $\frak{l}_\mu$ in Section~\ref{mainresultproof}, it is apparent that
\[
        \frak{l}_\mu = \bigg\{x \in \frak{g} \mid x_{ij} = 0 \textup{ if } i\frown k \in \eta \textup{ with } j<k \textup{ and } j \leq\frac{n+1}{2}\bigg\}.
\]
\end{proof}
  For a twisted $\mathbb{F}_q$-set partition $\eta$, we can write $\eta$ as a disjoint union of twisted $\mathbb{F}_q$-set partitions of the form $\{i\overset{a}{\frown}j \cup \bar{j} \overset{-a^q}{\frown}\bar{i}\}$ with $i \neq \bar{j}$ or of the form $\{i\overset{a}{\frown}\bar{i}\}$ with $a^q+a=0$. In other words, there exists $m$ such that
\[
        \eta = \bigsqcup_{r = 1}^m\eta_r
\]
with each $\eta_r$ of the described form. For $1 \leq r \leq m$, define $\lambda_r = \lambda_{\eta_r}$.

\begin{lem}\label{a2ulambda2} With notation as above,
\begin{enumerate}[label=\emph{(\alph*)}]
\item $\frak{u}_{\lambda_\eta} = \bigcap_{r=1}^m \frak{u}_{\lambda_r}$, and for any $s\geq 1$, $\frak{u} = \frak{u}_{s+1} +\bigcap_{r \leq s} \frak{u_r}$, and
\item $U_{\lambda_\eta} = \bigcap_{r=1}^m U_{\lambda_r}$, and for any $s\geq 1$, $U = U_{s+1}(\bigcap_{r \leq s}U_r)$.
\end{enumerate}
\end{lem}

\begin{proof} Part (a) follows directly from Lemma~\ref{a2ulambda}. Part(b) follows from (a) and the fact that $f(U_\mu) = \frak{u}_\mu$ for any $\mu \in\frak{u}^*$.
\end{proof}

For two characters $\chi$ and $\psi$, define their product by $(\chi\psi)(u) = \chi(u)\psi(u)$.

\begin{lem}\label{schprod} With notation as above,
\[
        \chi^\eta = \prod_{r=1}^m \chi^{\eta_r}.
\]
\end{lem}
\begin{proof} If $H_1$ and $H_2$ are subgroups of a finite group $G$ and $\psi_1$ and $\psi_2$ are characters of $H_1$ and $H_2$, respectively, then
\[
        \text{Ind}_{H_1}^G(\psi_1)\text{Ind}_{H_2}^G(\psi_2) = \sum_{x \in X} \text{Ind}_{{H_1}^x \cap H_2}^G (\psi_1^x\psi_2),
\]
where $X$ is a set of $(H_1,H_2)$ double coset representatives of $G$. In particular, if $HK = G$, then $\text{Ind}_{H_1}^G(\psi_1)\text{Ind}_{H_2}^G(\psi_2) = \text{Ind}_{{H_1}\cap H_2}^G (\psi_1\psi_2)$. By induction, if $H_1,...,H_k$ are subgroups of $G$ with $H_{s+1} (\bigcap_{r \leq s} H_r) = G$ for all $s \geq 1$, and $\psi_1,...,\psi_k$ are representations of the $H_r$, then
\[
         \prod_{r=1}^m\text{Ind}_{H_r}^G(\psi_r)= \text{Ind}_{(\bigcap_{r=1}^m H_r)}^G \prod_{r=1}^m \psi_r.
\]
The result follows from Lemma~\ref{a2ulambda2}.
\end{proof}

We now calculate the values of the supercharacters on the superclasses. First, we determine the dimensions of the elementary supercharacters.

\begin{lem} Let $\eta = \{i\overset{a}{\frown}j \cup \bar{j} \overset{-a^q}{\frown}\bar{i}\}$ (with $i \neq \bar{j}$) be a twisted $\mathbb{F}_q$-set partition; then
\begin{align*}
        \chi^\eta(1) = &|H \cdot \lambda_\eta| = \left\{\begin{array}{ll}
        q^{2(j-i-1)} & \quad \text{if $n$ is even,} \\
        q^{2(j-i-1)} & \quad \text{if $n$ is odd and $j \leq\frac{n+1}{2}$,} \\
        q^{2(j-i)} & \quad \text{if $n$ is odd and $j >\frac{n+1}{2}$.}
        \end{array}\right.
\end{align*}
Let $\eta = \{i\overset{a}{\frown}\bar{i}\}$ (with $i \leq \frac{n+1}{2}$) be a twisted $\mathbb{F}_q$-set partition; then
\begin{align*}
        \chi^\eta(1) = &|H \cdot \lambda_\eta| = \left\{\begin{array}{ll}
        q^{2(n-2i)} & \quad \text{if $n$ is even,} \\
        q^{2(n+1-2i)} & \quad \text{if $n$ is odd.}
        \end{array}\right.
\end{align*}
\end{lem}

\begin{proof} This follows from the fact that $|H \cdot \lambda_\eta| = |U:U_{\lambda_\eta}|$ and Lemma~\ref{a2ulambda}.
\end{proof}

We mention that the dimension of an arbitrary supercharacter can be calculated by applying Lemma~\ref{schprod}. Next we calculate the value of a supercharacter on a superclass.

\begin{thm} Let $\eta$ and $\nu$ be twisted $\mathbb{F}_q$-set partitions. Then
\[
        \chi^{\eta}(u_\nu) = \left\{\begin{array}{ll}
        \frac{\chi^{\eta}(1)}{(-q)^{\text{nst}_\nu^\eta}}
        \theta\bigg(\sum_{\substack{i\overset{a}{\frown}j \in \eta \\ i\overset{b}{\frown}j \in \nu}} ab\bigg) & \quad  \begin{array}{l}\text{if for }i \frown j \in \eta \text{ and } i<k<j, \\ i \frown k, k \frown j \notin \nu \end{array} \\
        0 & \quad \text{else,}\end{array}\right.
\]
where $\text{nst}_\nu^\eta = |\{i<j<k<l \mid j\frown k \in \nu,i\frown l \in \eta\}|$.
\end{thm}

\begin{proof} By Lemma~\ref{schprod}, the proof reduces to proving that the theorem holds in the case that $\chi^\eta$ is an elementary supercharacter. The technique we use is similar to that employed by Diaconis--Thiem in the proof of Theorem 5.1 of \cite{MR2491890}. First let $\eta = \{i\overset{a}{\frown}j \cup \bar{j} \overset{-a^q}{\frown}\bar{i}\}$ (with $i \neq \bar{j}$). We have that
\begin{align*}
        \chi^\eta(u_\nu)
        & = \chi^{\eta}(1)\bigg(\prod_{\substack{i< k <l < j \\ k\overset{b}{\frown}l \in \nu}}
        \frac{1}{q^4}\sum_{c_k,d_l \in \mathbb{F}_{q^2}} \theta(abc_kd_l+(abc_kd_l)^q)\bigg) \\
        &\cdot\bigg(\prod_{\substack{i< k <j \\ k\overset{b}{\frown}j \in \nu}}
        \frac{1}{q^2}\sum_{c_k \in \mathbb{F}_{q^2}} \theta(abc_k+(abc_k)^q)\bigg) \\
        & \cdot\bigg(\prod_{\substack{i< l <j \\ i\overset{b}{\frown}l \in \nu}}
        \frac{1}{q^2}\sum_{d_l \in \mathbb{F}_{q^2}} \theta(abd_l+(abd_l)^q)\bigg)\cdot \prod_{i\overset{b}{\frown} j \in \nu}\theta(ab+(ab)^q) \\
        & = \chi^{\eta}(1)\bigg(\prod_{\substack{i< k <l < j \\ k\overset{b}{\frown}l \in \nu}}
        \frac{1}{q^2}\bigg)
        \cdot\bigg(\prod_{\substack{i< k <j \\ k\overset{b}{\frown}j \in \nu}}
        0\bigg)
         \cdot\bigg(\prod_{\substack{i< l <j \\ i\overset{b}{\frown}l \in \nu}}
        0\bigg)\cdot \prod_{i\overset{b}{\frown} j \in \nu}\theta(ab+(ab)^q).
\end{align*}
It follows that
\[
        \chi^\eta(u_\nu) = \left\{\begin{array}{ll}
         \frac{\chi^\eta(1)}{(q^2)^{\#\{k \frown l \in \nu \mid i<k<l<j\}}}\prod_{i\overset{b}{\frown} j \in \nu}\theta(ab+(ab)^q) & \quad \text{if for } i<k<j \text{, } i \frown k,k \frown j \notin \nu, \\
         0 & \quad \text{else.}\end{array}\right.
\]
We can rewrite this as
\[
        \chi^\eta(u_\nu) = \left\{\begin{array}{ll}
        \frac{\chi^{\eta}(1)}{(-q)^{\text{nst}_\nu^\eta}}
        \theta\bigg(\sum_{\substack{r\overset{a}{\frown}s \in \eta \\ r\overset{b}{\frown}s \in \nu}} ab\bigg) & \quad \begin{array}{l}\text{if for }r \frown s \in \eta \text{ and } r<k<s, \\ r \frown k, k \frown s \notin \nu, \end{array} \\
        0 & \quad \text{else.}\end{array}\right.
\]

Now let $\eta = \{i\overset{a}{\frown}\bar{i}\}$ (with $i \leq \frac{n+1}{2}$). Then
\begin{align*}
        \chi^\eta(u_\nu)
        &= \chi^{\eta}(1)\bigg(\prod_{\substack{i< k <\frac{n}{2} \\k< l<\bar{k}\\ k\overset{b}{\frown}l \in \nu}}
        \frac{1}{q^2}\sum_{c_k,c_{\bar{l}} \in \mathbb{F}_{q^2}}\theta(abc_kc_{\bar{l}}^q+(abc_kc_{\bar{l}}^q)^q)\bigg) \\
        & \cdot\bigg(\prod_{\substack{i< l<\bar{i} \\ i\overset{b}{\frown}l \in \nu}}
        \frac{1}{q^2}\sum_{c_{\bar{l}} \in \mathbb{F}_{q^2}}\theta(abc_{\bar{l}}+(abc_{\bar{l}})^q)\bigg)\\
        &\cdot\bigg( \prod_{\substack{i < k \leq \frac{n}{2}\\k\overset{b}{\frown}\bar{k} \in \nu}}\frac{1}{q^2}\sum_{c_k \in \mathbb{F}_{q^2}}\theta(ab(c_k)^{q+1})\bigg)
        \cdot\prod_{i\overset{b}{\frown}\bar{i} \in \nu} \theta(ab) \\
        &= \chi^{\eta}(1)\bigg(\prod_{\substack{i< k <\frac{n}{2} \\k< l<\bar{k}\\ k\overset{b}{\frown}l \in \nu}}
        \frac{1}{q^2}\bigg)\bigg(\prod_{\substack{i< l<\bar{i} \\ i\overset{b}{\frown}l \in \nu}}
        0\bigg)
        \cdot \bigg(\prod_{\substack{i < k \leq \frac{n}{2}\\k\overset{b}{\frown}\bar{k} \in \nu}}\frac{1}{-q}\bigg)\prod_{i\overset{b}{\frown}\bar{i} \in \nu} \theta(ab).
\end{align*}
It follows that
\[
        \chi^\eta(u_\nu) = \left\{\begin{array}{ll}
         \frac{\chi^\eta(1)}{(-q)^{\#\{k \frown l \in \nu \mid i<k<l<\bar{i}\}}}\prod_{i\overset{b}{\frown} j \in \nu}\theta(ab+(ab)^q) & \quad \text{if for } i<k<j \text{, }i \frown k,k \frown j \notin \nu, \\
         0 & \quad \text{else.}\end{array}\right.
\]
This can be rewritten as
\[
         \chi^\eta(u_\nu) = \left\{\begin{array}{ll}
        \frac{\chi^{\eta}(1)}{(-q)^{\text{nst}_\nu^\eta}}
        \theta\bigg(\sum_{\substack{r\overset{a}{\frown}s \in \eta \\ r\overset{b}{\frown}s \in \nu}} ab\bigg) & \quad \begin{array}{l}\text{if for }r \frown s \in \eta \text{ and } r<k<s, \\ r \frown k, k \frown s \notin \nu, \end{array} \\
        0 & \quad \text{else.}\end{array}\right.
\]
\end{proof}

Note that in the above formula the sum of the terms $ab$ is an element of $\mathbb{F}_q$, even though each individual term might not be. This formula is identical to the type $A$ supercharacter formula (which was first derived by Andr\'e, and can be found in \cite{MR2592079}), except with $-q$ replacing $q$ everywhere (note that the supercharacter degrees are all powers of $q^2$). This idea that information about the unitary group can be obtained from information about the general linear group by replacing $q$ with $-q$ is referred to as \emph{Ennola duality} (see \cite{MR803335}).

\section{Acknowledgements}
I would like to thank Nat Thiem for his numerous insights and suggestions which contributed to the quality of this paper. Many thanks as well to the referee for the helpful comments.

\bibliography{bibfile}
	\bibliographystyle{plain}

\end{document}